\documentclass[a4paper,12pt]{amsart}

\usepackage[latin1]{inputenc}
\usepackage[T1]{fontenc}
\usepackage[english]{babel}

\usepackage{mathrsfs}
\usepackage{amscd}
\usepackage{amsfonts}
\usepackage{amsmath}
\usepackage{amssymb}
\usepackage{amstext}
\usepackage{amsthm}
\usepackage{amsbsy}

\usepackage{xspace}
\usepackage[all]{xy}
\usepackage{graphicx}
\usepackage{url}
\usepackage{latexsym}

\makeatletter
\newcommand*{\rom}[1]{\expandafter\@slowromancap\romannumeral #1@}
\makeatother

\theoremstyle{definition}

\newtheorem{fact}{fact}

\newtheorem{thm}[fact]{Theorem}
\newtheorem{lemma}[fact]{Lemma}
\newtheorem{prop}[fact]{Proposition}
\newtheorem{corollary}[fact]{Corollary}
\newtheorem{defini}[fact]{Definition}
\newtheorem{rem}[fact]{Remark}

\title{The Lost Melody Phenomenon}
\author{Merlin Carl}

\begin{document}

\maketitle

\begin{abstract}
A typical phenomenon for machine models of transfinite computations is the existence of so-called lost melodies, i.e. real numbers $x$ such that
the characteristic function of the set $\{x\}$ is computable while $x$ itself is not (a real having the first property is called recognizable). This was first observed by J. D. Hamkins and A. Lewis for infinite 
time Turing machine (see \cite{HaLe}), then demonstrated by P. Koepke and the author
for $ITRM$s (see \cite{ITRM}). We prove that, for unresetting infinite time register machines introduced by P. Koepke in \cite{wITRM}, recognizability equals computability, i.e. the lost melody phenomenon does not occur.
Then, we give an overview on our results on the behaviour of recognizable reals for $ITRM$s as introduced in \cite{KoMi}. We show that there are no lost melodies for ordinal Turing machines ($OTM$s)
or ordinal register machines ($ORM$s) without parameters and that this is, under the assumption that $0^{\sharp}$ exists, independent from $ZFC$.
Then, we introduce the notions of resetting and unresetting $\alpha$-register machines and give some information on the question for which of these machines there are lost melodies.
\end{abstract}

\section{Introduction}

The research on machine models of transfinite computations began with the seminal Hamkins-Lewis paper \cite{HaLe} on Infinite Time Turing Machines ($ITTM$s). These machines, which are basically classical Turing
machines equipped with transfinite running time, have succesfully been applied to various areas of mathematics such as descriptive set theory (\cite{Co}, \cite{SeSc}) and model theory (\cite{HMSW}) and turned out to show a variety of fascinating
behaviour. A particularly interesting feature that has frequently played a role in applications is the existence of so-called lost melodies. A lost melody is a real number $x\subseteq\omega$ which is recognizable, i.e. for some $ITTM$-program $P$, 
the computation of
$P$ with $y$ on the input tape (which plays the role of a real oracle for an $ITTM$) is defined for all $y$ and outputs $1$ iff $y=x$ and otherwise outputs $0$, but not computable, i.e. no program computes 
the characteristic function of $x$. The existence of lost melodies for $ITTM$s was observed and proved
in \cite{HaLe}.

In the meantime, a rich variety of transfinite machine types have been defined, studied and related to each other: Koepke introduced Infinite Time Register Machines (see \cite{wITRM}), which were later relabeled 
as unresetting or weak Infinite Time Register Machines ($wITRM$s) when an enhanced version was considered in \cite{ITRM}. Further generalizations led to $\alpha$-Turing machines (\cite{KoSe1}) $\alpha$-$\beta$-Turing machines, 
transfinite $\lambda$-calculus (\cite{Se}),
the hypermachines of Friedman and Welch (basically $ITTM$s with a more complex limit behaviour, see \cite{FrWe}) and infinite time Blum-Shub-Smale-machines (\cite{KoSe2}). An arguably ultimate upper bound is set by Koepke's
ordinal register machines ($ORM$s) and ordinal Turing machines ($OTM$s), which, using ordinal parameters, can calculate the whole of G\"odel's constructible hierarchy $L$. (\cite{ICTT} contains an argument to the effect that $OTM$-computability
is indeed a conceptual analogue of Turing-computability in the transfinite.) For many of these machine types, the computational strength has been precisely determined.

In this paper, we are interested in the question how typical the existence of lost melodies is for models of transfinite computations. While it was shown in \cite{ITRM} that $ITRM$s, like $ITTM$s, have lost melodies, the question was to the best
of our knowledge not considered for any other of these machine types and has in particular been open concerning $wITRM$s. Specifically, we focus on machine models generalizing register machines: In section $1$, we prove that there are no lost melodies for
unresetting $ITRM$s, we summarize (mostly leaving out or merely sketching proofs) in section $2$ some of our earlier results on $ITRM$-recognizability obtained in \cite{Ca} and \cite{Ca2} and proceed in section $3$ to show that there are again no lost melodies for 
ordinal register- and Turing machines, without ordinal
parameters and that the answer for ordinal machines with parameters is undecidable under a certain set-theoretical assumption.
 Then, for the parameter-free case, we interpolate between these extrem cases by introducing resetting and unresetting $\alpha$-register machines and show that for resetting $\alpha$-register machines, lost melodies always exist. For unresetting $\alpha$-register machines,
the picture is quite different: It turns out that, while there are no lost melodies for $\alpha=\omega$, there exist countable values of $\alpha$ for which there are, but their supremum is countable, so that from some $\gamma<\omega_1$ on, 
lost melodies for unresetting $\alpha$-register machines cease to exist.

Let us now introduce the relevant machine types, the resetting and unresetting $\alpha$-register machines. (The unresetting version was originally suggested in the final paragraph of \cite{wITRM}.)
An $\alpha$-register machine has finitely many registers, each of which can store a single ordinal $<\alpha$.
The instructions for an $\alpha$-register machine (also simply called $\alpha$-machine) are the same as for
the unlimited register machines of \cite{Cu}: the increasing of a register content by $1$, copying a register content to another register, reading out the $r_{i}$th bit of an oracle (where $r_{i}$
is the content of the $i$th register), jumping to a certain program line provided a certain register 
content is $0$, and stopping. Programs for $\alpha$-register machines are finite sequences of instructions, as usual. The running time of an $\alpha$-machine is the class of ordinals. At successor times, computations proceed as for the classical
model of unlimited register machines, introduced in \cite{Cu}.
It remains to fix what to do at a limit time $\lambda$. We consider three possibilites, where $Z_{\iota}$ denotes the active program line at time $\iota$ and $R_{i\iota}$ denotes the content of the $i$th register at time $\iota$:

\begin{itemize}
 \item $Z_{\lambda}$ and $R_{i\lambda}$ are undefined. Setting $\lambda=\omega$, this would just be a classical $URM$
 \item $Z_{\lambda}:=\liminf_{\iota<\lambda}Z_{\iota}$, $R_{i\iota}=\liminf_{\iota<\lambda}R_{i\iota}$, if the latter is $<\alpha$ and otherwise, the computation is undefined. Setting $\alpha=\omega$, these are
the unresetting or weak infinite time register machines introduced in \cite{wITRM}.\footnote{In the cited paper, these machines are just called infinite time register machines, without further qualification. Later on, when resetting
infinite time register machines were introduced, the terminology was changed.} We call these unresetting or weak $\alpha$-machines.
\item $Z_{\lambda}:=\liminf_{\iota<\lambda}Z_{\iota}$, $R_{i\iota}=\liminf_{\iota<\lambda}R_{i\iota}$, if the latter is $<\alpha$ and otherwise $R_{i\lambda}=0$. Setting $\alpha=\omega$, these are the infinite time register machines ($ITRM$s)
of \cite{KoMi}. We call these resetting or strong $\alpha$-machines.
\end{itemize}

Most of our notation and terminology is standard. $KP$ is Kripke-Platek set theory (see e.g. \cite{Sa}), $ZF^{-}$ is Zermelo-Fraenkel set theory without the power set axioms in the version described in \cite{GHJ}.
If $P$ is a program and $x$ a real, then $P^{x}\downarrow$ means that $P$, when run in the oracle $x$, stops, while $P^{x}\uparrow$ means that $P$ in the oracle $x$ diverges. $P^{x}\downarrow=y$
means that $P^{x}(i)\downarrow$ for all $i\in\omega$ and that in the final state of $P^{x}(j)$, the first register contains $1$ iff $j\in y$ and otherwise $0$.
We write $x\leq_{h}y$ for hyperarithmetic reducibility, i.e. for $x\in L_{\omega_{1}^{CK,y}}[y]$. $On$ denotes the class of ordinals, small greek letters denote ordinals unless stated otherwise. $p:On\times On\rightarrow On$
is Cantor's pairing function.

It turns out (see \cite{KoMi}) that unresetting $\omega$-machines are much weaker than their resetting analogue; in particular, resetting $\omega$-machines can compute all finite iterations of the halting problem for unresetting $\omega$-machines.
We fix the following general definitions:

\begin{defini}
 Let $P$ be program of any of the machine types described above, and let $x$ be a real. We say that $P$ recognizes $x$ iff $P^{x}\downarrow=1$ and $P^{y}\downarrow=0$ for all $y\neq x$.
We say that $x$ is recognizable by an (un)resetting $\alpha$-machine iff there is a program $P$ for such a machine that recognizes $x$. When the machine type is clear from the context, 
we merely state that $x$ is recognizable.
\end{defini}

\section{Weak ITRMs}

\begin{prop}{\label{trivialdirection}}
 Let $x$ be $wITRM$-computable. Then $x$ is $wITRM$-recognizable.
\end{prop}
\begin{proof}
Let $P$ be a $wITRM$-program that computes $x$. The idea is to compare $x$ to the oracle bitwise. A bit of care is necessary to arrange this comparison without
overflowing registers. Use a separate counting register $R$ and two flag registers $R^{flag}_{1}$ and $R^{flag}_{2}$. Initially, $R$ and $R^{flag}_{1}$ contain $0$ and $R^{flag}_{2}$ contains $1$. 
In a computation step, when $R$ contains $i$, compute the $i$th
bit of $x$ and compare it to the oracle. If these bits disagree, we stop with output $0$. Otherwise, we successively set all registers but $R^{flag}_{1}$ and $R^{flag}_{2}$ to $0$ once and then
set the content of $R$ to $i+1$ (after a register has been set to $0$, it may be used to store $i$ for this purpose) and swap the contents of $R^{flag}_{1}$ and $R^{flag}_{2}$.
In this way, if the number in the oracle is $x$, then a state will eventually occur in which $R^{flag}_{1}$ and $R^{flag}_{2}$ both contain $0$ and $R$ contains $0$, in which case we
output $1$.
\end{proof}

\begin{defini}
Let us denote by $wRECOG$ the set of reals recognizable by a weak $ITRM$ and by $wCOMP$ the set of reals computable by a weak $ITRM$.
\end{defini}

The following is the relativized version of Theorem $1$ of \cite{wITRM}:

\begin{thm}{\label{relwITRM}}
 Let $x,y\subseteq\omega$. Then $x$ is $wITRM$-computable in the oracle $y$ iff $x\in L_{\omega_{1}^{CK,y}}[y]$. In particular, $x$ is $wITRM$-computable iff $x\in L_{\omega_{1}^{CK}}$ iff
$x$ is hyperarithmetic.
\end{thm}
\begin{proof}
The proof given in \cite{ITRM} relativizes. We omit the proof to avoid what would amount to a mere repition of that proof.
\end{proof}

\begin{lemma}{\label{relbar}}
Let $x\subseteq\omega$ and let $M\models KP$ be such that $\omega^{M}=\omega$ and $x\in M$. Then $\omega_{1}^{CK,x}$ is an initial segment of $On^{M}$.
\end{lemma}
\begin{proof}
If $x\in M$, then, as $M\models KP$, we have $z\in M$ for every $z$ which is recursive in $x$.

Now let $x\in M$. We have to show that every $\alpha<\omega_{1}^{CK,x}$ belongs to the well-founded part of $M$.
Since $M\models KP$, $M$ satisfies the recursion theorem for $\Sigma_{1}$-definitions. Let $z\subset\omega\times\omega$ be such that
$(\omega,z)$ is a well-ordering. For all $\beta\in On$, we define, by $\Sigma_{1}$-recursion, a function $F$ via $F(\beta)=\text{sup}_{z}\{F(\gamma)+1|\gamma<\beta\}$ if
this supremum exists, and otherwise $F(\beta)=\omega$.

We show that $\text{rng}(F)=\omega$: Otherwise, we have $\text{rng}(F)\subsetneq\omega$ and $\text{rng}(F)$ is closed under $z$-predecessors. Since $(\omega,z)$ is a well-ordering in $V$, $\text{rng}(F)$ must have a $z$-supremum $n\in\omega$.
Hence\\ $\text{rng}(F)=\{m\in\omega|(m,n)\in z\}\in M$: By the injectivity of $F$, $F^{-1}$ is $\Sigma_{1}$-definable. By $\Sigma_{1}$-replacement, $\text{rng}(F^{-1})$ is a set, hence an ordinal $\gamma$.
Consequently, we have $\omega\subseteq \text{rng}(F)$. We now show that $\text{rng}(F|\gamma)=\omega$ for some $\gamma\in On\cap M$. Suppose that $\omega\notin \text{rng}(F)$. Then $F$ is injective and $F^{-1}:\omega\rightarrow On$
is a function, contradicting $\Sigma_{1}$-replacement (as $On\cap M$ is not a set in $M$). Relativizing this argument to $x$, we obtain the desired result.
\end{proof}

\begin{lemma}{\label{cyclecrit}}
Let $P$ be a $wITRM$-program, and let $x\subseteq\omega$. Then $P^{x}\uparrow$ iff there exist $\sigma<\tau<\omega_{1}^{CK,x}$ such that $Z(\tau)=Z(\sigma)$, $R_{i}(\tau)=R_{i}(\sigma)$ for all $i\in\omega$
and $R_{i}(\gamma)\geq R_{i}(\sigma)$ for all $i\in\omega$, $\sigma<\gamma<\tau$. (Here, $Z(\gamma)$ and $R_{i}(\gamma)$ denote the active program line and the content of register $i$ at time $\gamma$.)
\end{lemma}
\begin{proof}
This is an easy adaption of Lemma $3$ of \cite{KoMi}.
\end{proof}

The following lemma allows us to quantify over countable $\omega$-models of $KP$ by quantifying over reals:

\begin{lemma}{\label{kpsigma1}}
 There is a $\Sigma_{1}^{1}$-statement $\phi(v)$ such that $\phi(x)$ holds only if $x$ codes an $\omega$-model of $KP$ and such that, for any countable $\omega$-model $M$ of $KP$, there is a code $c$ for $M$
such that $\phi(c)$ holds.
\end{lemma}
\begin{proof}
Every countable $\omega$-model $M$ of $KP$ can be coded by a real $c(M)$ in such a way that the $i\in\omega$ is represented by $2i$ in $c$ and $\omega$ is represented by $1$.
We can then consider a set $S$ of statements saying that a real $c$ codes a model of $KP$ together with $\{P_{k}|k\in\omega+1\}$, where $P_{k}$ is the statement $\forall{i<k}(p(2i,2k)\in c)\wedge\forall{j}\exists{i<k}(p(j,2k)\in{c}\rightarrow j=2i)$ for 
$k\in\omega$ and $P_{\omega}$ is the statement $\forall{i}(p(2i,1)\in c)\wedge\forall{j}\exists{i}(p(j,1)\in c\rightarrow j=2i)$.
Then $\bigwedge{S}$ is a hyperarithmetic conjunction of arithmetic formulas in the predicate $c$. But such a conjunction is equivalent to a $\Sigma_{1}^{1}$-formula.

\end{proof}

\begin{thm}{\label{wsigma1}}
Let $x$ be recognizable by a $wITRM$. Then $\{x\}$ is a $\Sigma_{1}^{1}$-singleton.
\end{thm}
\begin{proof}
Let $P$ be a program that recognizes $x$ on a $wITRM$. Let $KP(z)$ be a $\Sigma_{1}^{1}$-formula (in the predicate $z$) stating that $z$ codes an $\omega$-model of $KP$ with $\omega$ represented by $1$
and every integer $i$ represented by $2i$ as constructed in Lemma \ref{kpsigma1}. 
Let $E(y,z)$ be a first-order formula (in the predicates $y$ and $z$) stating that the structure coded by $z$ contains $y$. (We can e.g. take $E(y,z)$ to be $\exists{k}\forall{i}(z(i)\leftrightarrow z(p(2i,k)))$.)
Furthermore, let $Acc_{P}(z,y)$ be a first-order formula (in the predicates $y$ and $z$) stating that $P^{y}\downarrow=1$ in the structure coded by $z$. 
Finally, let $NC_{P}(y)$ be a first-order formula (in the predicate $y$) stating that in the computation $P^{y}$, there are no two states $s_{\iota_{1}},s_{\iota_{2}}$ with $\iota_{1}<\iota_{2}$
such that $s_{\iota_{1}}=s_{\iota_{2}}$ and, for every $\iota_1<\iota<\iota_2$, the content $r_{i\iota}$ of register $R_i$ at time $\iota$ is at least $r_{i\iota_{1}}$ (the content of $R_i$ at time $\iota_{1}$)
and the index of the active program at time $\iota$ is not smaller than the index of the active program line at time $\iota_{1}$. (This is just the cycle criterion from Lemma \ref{cyclecrit}.)

This is possible in $KP$ models containing $x$ since, by Lemma \ref{relbar} above, $\omega_{1}^{CK,x}$ is an initial segment of the well-founded part of
each such model and, by Lemma \ref{cyclecrit}, the computation either cycles before $\omega_{1}^{CK,x}$ or stops - thus the cycling or halting behaviour takes part in the well-founded part of the model and is hence
absolute between such a model and $V$.
Now, take $\phi(a)$ to be $\exists{z}(KP(z)\wedge E(a,z)\wedge Acc_{P}(z,a)\wedge NC_{P}(a))$. This is a $\Sigma_{1}^{1}$-formula. We claim that $x$ is the only solution to $\phi(a)$:
To see this, first note that $x$ clearly is a solution, since $\omega_{1}^{CK,x}$ is an initial segment of every $KP$-model containing $x$ by Lemma \ref{relbar}.

On the other hand, assume that $b\neq x$. In this case, as $P$ recognizes $x$, we have $P^{b}\downarrow=0$ in the real world, and hence, by absoluteness of $wITRM$-(oracle)-computations for $KP$-models containing
the relevant oracles, also inside $L_{\omega_{1}^{CK,b}}[b]$. Now $L_{\omega_{1}^{CK,b}}[b]$ is certainly a countable $KP$-model containing $b$, hence a counterexample to $\phi(b)$. Hence $\phi(b)$ is false if $b\neq x$, as desired.
\end{proof}

\begin{corollary}
If a real $x$ is $wITRM$-recognizable, then it is $wITRM$-computable. Hence, there are no lost melodies for weak $ITRM$s.
\end{corollary}
\begin{proof}
By Kreisel's basis theorem (see \cite{Sa}, p. $75$), if $a$ is not hyperarithmetical and $B\neq\emptyset$ is $\Sigma_{1}^{1}$, then $B$ contains some element $b$ such that $a\nleq_{h}b$. 
Now suppose that $x$ is $wITRM$-recognizable. By Theorem \ref{wsigma1}, $\{x\}$ is $\Sigma_{1}^{1}$ and certainly non-empty. If $x$ was not hyperarithmetical, then, by Kreisel's theorem, $\{x\}$ would
contain some $b$ such that $x\nleq_{h}b$. But the only element of $\{x\}$ is $x$, so $x\notin HYP$ implies $x\nleq_{h}x$, which is absurd. Hence $x\in HYP$. So $x$ is $wITRM$-computable.
\end{proof}

\section{Resetting ITRMs}

$ITRM$-recognizability was considered in \cite{ITRM}, \cite{Ca} and \cite{Ca2}. We give here a summary of some of the most important results.
Recall that an $ITRM$ is different from a $wITRM$ in that it, in case of a register overflow, resets the content of the overflowing registers to $0$ and continues computing.

The following characterization of the computational strength of $ITRM$s with real oracles is a relativized version of the main theorem of \cite{KoMi}:

\begin{thm}
 $x$ is $ITRM$-computable in the oracle $y$ iff $x\in L_{\omega_{\omega}^{CK,y}}[y]$.
\end{thm}

We saw that computability equals recognizability for $wITRM$s. For $ITRM$s, the situation is very different. Clearly, analogous to Proposition \ref{trivialdirection}, the computable reals are still recognizable. 
But, for $ITRM$s, the lost melody phenomenon does occur:

\begin{thm}{\label{lostmelody}}
 There exists a real $x$ such that $x$ is not $ITRM$-computable, but $ITRM$-recognizable.
\end{thm}
\begin{proof}
$x$ can be taken to be a $<_{L}$-minimal real coding an $\in$-minimal $L_{\alpha}$ such that $L_{\alpha}\models ZF^{-}$. See \cite{ITRM} for the details. 
\end{proof}

\begin{rem} There are more straightforward examples. In \cite{Ca2}, it is shown that, if $(P_{i}|i\in\omega)$ is some natural enumeration of the $ITRM$-programs,
then $h:=\{i\in\omega|P_{i}\downarrow\}$, the halting number for $ITRM$s, is recognizable. The usual argument of course shows that it is not $ITRM$-computable. \end{rem}

\begin{defini}
 $RECOG$ denotes the set of $ITRM$-recognizable reals.
\end{defini}

\begin{rem} Write $RECOG_n$ for the set of reals that are recognizable by an $ITRM$ using at most $n$ registers. It was shown in \cite{Ca} that
$RECOG_{n}\subsetneq RECOG$, i.e. the recognizability strength of $ITRM$s increases with the number of registers. This corresponds to the result 
established in \cite{KoMi} that the computational strength of $ITRM$s increases with the number of registers.\end{rem}

Using Shoenfield's absoluteness
lemma, it is not hard to see that recognizable reals are always constructible (see \cite{Ca}). We consider
the distribution of recognizable reals in the canonical well-ordering $<_{L}$ of the constructible universe:

\begin{thm}
There are gaps in the $ITRM$-recognizable reals, i.e. there are $x,y,z\in\mathcal{P}(\omega)\cap L$ such that $x<_{L}y<_{L}z$, $x,z\in RECOG$, but $y\notin RECOG$.
\end{thm}
\begin{proof}
 As there are only countably many $ITRM$-recognizable real, there must exist a countable $\alpha$ such that $L_{\alpha}\models ZF^{-}$ and $L_{\alpha}$ contains
some non-recognizable reals $y$. Let $z$ be the $<_{L}$-minimal code of the $\in$-minimal $L_{\alpha}$ with these properties and let $x=0$. Then $x$, $y$ and $z$
are as desired. The details can be found in \cite{Ca}.
\end{proof}

This suggest the detailed study of the distribution of the $ITRM$-recognizable reals among the constructible reals, which was carried out in \cite{Ca} and \cite{Ca2}.
We summarize the main results.

\begin{defini}
 Let $\alpha\in On$. $\alpha$ is called $\Sigma_1$-fixed iff there exists a $\Sigma_1$-formula $\phi$ such that $\alpha$ is minimal with $L_{\alpha}\models\phi$.
We also let\\ $\sigma:=\text{sup}\{\alpha|\alpha\text{ is }\Sigma_{1}-\text{fixed}\}$.
\end{defini}

\begin{rem} It is easy to see by reflection that the $\Sigma_1$-fixed ordinals are countable and that there are countably many of them (as there are only countable many formulas). Hence $\sigma$ is countable as well.
It can also be shown that $\sigma$ is the supremum of parameter-free $OTM$-halting times. \end{rem}

\begin{thm}
 (a) $RECOG\subseteq L_{\sigma}$\\
 (b) $\{\alpha|RECOG\cap(L_{\alpha+1}-L_{\alpha})\neq\emptyset\}$ is cofinal in $\sigma$\\
 (c) For every $\gamma<\sigma$, there exists $\alpha<\sigma$ such that\\ $RECOG\cap(L_{\alpha+\gamma}-L_{\alpha})=\emptyset$.
\end{thm}
\begin{proof}
 See \cite{Ca}.
\end{proof}

The $ITRM$-computability of a real can be effectively characterized in purely set theoretical terms (namely as being an element of $L_{\omega_{\omega}^{CK}}$). Correspondingly,
we have the following necessary criterion for $ITRM$-recognizability:

\begin{lemma}{\label{bigadmissibles}}
 Let $x\in RECOG$. Then $x\in L_{\omega_{\omega}^{CK,x}}$. In particular, we have $\omega_{\omega}^{CK,x}>\omega_{\omega}^{CK}$, hence $\omega_{i}^{CK,x}>\omega_{i}^{CK}$ for some $i\in\omega$.
\end{lemma}
\begin{proof}
 See \cite{Ca2}.
\end{proof}

Lemma \ref{bigadmissibles} in fact allows a machine-independent characterization of recognizability:

\begin{thm}
 Let $x\in\mathcal{P}^{L}(\omega)$. Then $x\in RECOG$ iff $x$ is the unique witness for some $\Sigma_{1}$-formula in $L_{\omega_{\omega}^{CK,x}}$.
\end{thm}
\begin{proof}
 See \cite{Ca2}. 
\end{proof}

One might now ask where non-recognizable occur; clearly, every real in $L_{\omega_{\omega}^{CK}}$ is recognizable, but what happens above $\omega_{\omega}^{CK}$? E.g., is there some $\alpha>\omega_{\omega}^{CK}$
such that the reals in $L_{\alpha}$ are still all recognizable? It turns out that this is not the case and that, in fact, unrecognizables turn up wherever possible in the $L$-hierarchy.

\begin{defini}
 $\alpha\in On$ is an index iff $(L_{\alpha+1}-L_{\alpha})\cap\mathcal{P}(\omega)\neq\emptyset$.
\end{defini}

\begin{thm}
 Let $\alpha\geq\omega_{\omega}^{CK}$ be an index. Then there exists a real $x\notin RECOG$ such that $x\in L_{\alpha+1}-L_{\alpha}$.
\end{thm}
\begin{proof} See \cite{Ca2}.

\end{proof}

In the light of Lemma \ref{bigadmissibles}, it is natural to concentrate the study of recognizability on reals $x$ with $x\in L_{\omega_{\omega}^{CK,x}}$. It turns out that the distribution of
recognizables becomes much tamer when we do this:

\begin{thm}(The `All-or-nothing-theorem')
 Let $\gamma$ be an index. Then either all $x\in L_{\gamma+1}-L_{\gamma}$ with $x\in L_{\omega_{\omega}^{CK,x}}$ are recognizable or none of them is.
\end{thm}
\begin{proof}
 See \cite{Ca2}. The idea is that, given a recognizable $a\in L_{\gamma+1}-L_{\gamma}$, this can be used to identify the $<_{L}$-minimal code $c$ of $L_{\gamma+1}$, which can in turn be used to identify
every real in $L_{\gamma+1}$. 
\end{proof}

\section{Ordinal Machines}

Ordinal Turing machines ($OTM$s) and ordinal register machines ($ORM$s) were introduced in \cite{OTM} and \cite{ORM}, respectively, and seem to provide an upper bound on the strength of a reasonable transfinite model of computation. 
(See e.g. \cite{ICTT} for an argument in favor of this claim.) In the papers just cited, Koepke proves that, when finite sets of ordinals are allowed as parameters, these machines can compute the characteristic function of a set $x$ of ordinals iff $x\in L$.
In particular a real $x$ is computable by such a machine iff $x\in L$. We formulate our results from now on for $OTM$s only, as they carry over verbatim to $ORM$s. To clarify the role of the parameters, we give a separate definition for 
recognizability by $OTM$s with parameters.

\begin{defini}
 $x\subseteq\omega$ is parameter-$OTM$-recognizable iff, for some $OTM$-program $P$ with a finite sequence $\vec{\gamma}$ of ordinal parameters
 and every $y\subseteq\omega$, $P^{y}\downarrow=1$ iff $x=y$ and otherwise $P^{y}\downarrow=0$.
\end{defini}

In the constructible universe, there are no lost melodies for parameter-$OTM$s:

\begin{lemma}{\label{OTMconstructible}}
Assume that $V=L$ and let $x$ be parameter-$OTM$-recognizable. Then $x$ is parameter-$OTM$-computable.
\end{lemma}
\begin{proof}
By \cite{Ko}, a set $S$ of ordinals is computable by an $OTM$-program with ordinal parameters iff $S\in L$.
Hence, every constructible real is parameter-$OTM$-computable, and in particular each parameter-$OTM$-recognizable real.\end{proof}

Note that, of course, every constructible real is also parameter-$OTM$-recognizable.

\begin{lemma}{\label{countableparamters}}
Assume that $\omega_{1}^{L}=\omega_{1}$. Let $\gamma<\omega_{1}^{L}$ and suppose that $x\subseteq\omega$ is recognizable by some program $P$ in the parameter $\gamma$. Then $x\in L$.
\end{lemma}
\begin{proof}
 As $\gamma$ is countable in $L$, there is a constructible real $z$ coding $\gamma$. Pick $z$ $<_{L}$-minimal. Then $\exists{y}P^{y}(\gamma)\downarrow=1$ is
expressible by a $\Sigma_1$-formula in the parameter $z$. Let $\rho$ be the running time of $P^{x}(\gamma)$ (i.e. the length of the computation). Then $\rho$ is countable:
To see this, let $c$ be the computation of $P^{x}(\gamma)$, $\kappa$ a cardinal in $L[x]$ such that $\kappa>\text{max}\{|c|,\aleph_{1}^{L[x]}\}$ and consider in $L_{\kappa}[x]$ the $\Sigma_{1}$-Skolem
hull $H$ of $\{c,\gamma,x\}$. By condensation in $L[x]$, there is some $\bar{\kappa}$ such that $H$ collapses to $L_{\bar{\kappa}}[x]$; let $\pi:H\rightarrow L_{\bar{\kappa}}[x]$ be the collapsing map.
Moreover, as $H$ is countable, so is $L_{\bar{\kappa}}[x]$. As $x\subseteq\omega\subseteq H$, we have $\pi(x)=x$. As $\omega+1\subseteq H$, $\gamma\in H$ and $H$ contains a bijection between $\omega$ and $\gamma$,
we have $\gamma\subseteq H$, so $\pi(\gamma)=\gamma$. As `$c$ is the computation of $P^{x}(\gamma)$' is expressible by a $\Sigma_{1}$-formula, $L_{\bar{\kappa}}[x]$ believes that $\pi(c)$ is the computation
of $P^{\pi(x)}(\pi(\gamma))$. As $L_{\bar{\kappa}}[x]$ is transitive and by absoluteness of computations, $\pi(c)$ really is the computation of $P^{\pi(x)}(\pi(\gamma))$. As $\pi(x)=x$ and $\pi(\gamma)=\gamma$,
we have $\pi(c)=c$, so $c\in L_{\bar{\kappa}}[x]$; as the latter is transitive and countable, $c$ is countable. Hence $\rho$ is countable.\\

As $\rho<\omega_{1}=\omega_{1}^{L}$, $\rho$ is countable in $L$. As there are cofinally in $\omega_{1}^{L}$ many admissible ordinals, 
let $\alpha>\text{max}\{\gamma,\rho\}$ be a limit of admissible ordinals which is also a limit or index ordinals such that
$z\in L_{\alpha}$. Now $\exists{y}P^{y}(\gamma)\downarrow=1$ is expressible as a $\Sigma_{1}$-formula $\phi(z)$ in the real parameter $z$. As $\rho<\alpha$ and $x\in L_{\alpha}[x]$,
$\phi(z)$ holds in $L_{\alpha}[x]$ and hence in $V_{\alpha}$. By a theorem of Jensen and Karp (see section $5$ of \cite{JeKa}), $\Sigma_{1}$-formulas are absolute between $L_{\zeta}$ and $V_{\zeta}$ when
$\zeta$ is a limit of admissibles and $L_{\zeta}$ contains the relevant parameters. Hence $\phi(z)$ holds in $L_{\alpha}$. So $L_{\alpha}$ contains a real $y$ such that, in $L_{\alpha}$, we have $P^{y}(\gamma)\downarrow=1$.
By absoluteness of computations, we have $P^{y}(\gamma)\downarrow=1$ in the real world. As $x$ is recognized by $P$ in the parameter $\gamma$, it follows that $y=x$. Hence $x\in L$.
\end{proof}

On the other hand, if the universe is much unlike $L$ and we allow uncountable parameters, lost melodies for parameter-$OTM$s can occur:

\begin{thm}{\label{0sharprecog}}
 Assume that $0^{\sharp}$ exists. Then there is a lost melody for parameter-$OTM$s. In fact, $0^{\sharp}$ is parameter-$OTM$-recognizable in the parameter $\omega_{1}$.
\end{thm}
\begin{proof}
By Theorem $14.11$ of \cite{Ka}, the relation $x=0^{\sharp}$ is $\Pi_{2}^{1}$, so $x\neq0^{\sharp}$ is $\Sigma_{2}^{1}$. Furthermore, $\Sigma_{2}^{1}$-relations are absolute between
transitive models of $KP$ containing $\omega_{1}$ (see e.g. Corollary $1$ of \cite{SeSc}). 
Now, let $\alpha>\omega_{1}$ be minimal such that $M:=L_{\alpha}[0^{\sharp}]\models KP$. Then $L[0^{\sharp}]$ contains a bijection $f:\omega_{1}\leftrightarrow M$.
Hence, $M$ is coded by $r:=\{p(\iota_{1},\iota_{2})\mid \iota_{1},\iota_{2}<\omega_{1}\wedge f(\iota_{1})\in f(\iota_{2})\}\in L[0^{\sharp}]$. 
To recognize $0^{\sharp}$ with an $OTM$ when $\omega_1$ is given as a parameter, 
we proceed as follows: Given a real $x$ in the oracle, search through the 
subsets of $\omega_{1}$ in $L[x]$ (by a similar procedure used in the proof of Theorem \ref{noOTMLoMe}) for 
a set $c$ coding a $KP$-model $M^{\prime}$ of the form $L_{\beta}[x]$
that contains $\omega_{1}$. As such sets exist in $L[x]$, such a $c$
 will eventually be found. Once this has happened, check, using $c$, whether $M^{\prime}\models x=0^{\sharp}$. If not, then, by absoluteness, $x\neq 0^{\sharp}$, otherwise $x=0^{\sharp}$.
\end{proof}

Taken together, the last two theorems readily yield:

\begin{corollary}{\label{LoMeindependence}}
 If $0^{\sharp}$ exists, then it is undecidable in $ZFC$ whether there are lost melodies for parameter-$OTM$s.
\end{corollary}

From now on, when we talk about $OTM$s, we always mean the parameter-free case.
What happens if we consider $OTM$s without ordinal parameters? It turns out that then, there are no lost melodies:

\begin{thm}{\label{noOTMLoMe}}
 Let $x\subseteq\omega$ and $P$ be an $OTM$-program such that, for each $y\subseteq\omega$, we have $P^{y}\downarrow=1$ iff $y=x$ and $P^{y}\downarrow=0$, otherwise. Then $x$ is $OTM$-computable (without parameters).
\end{thm}
\begin{proof}
In \cite{Ko}, it is shown that every constructible set of ordinals is uniformly computable from an appropriate finite set of ordinal parameters. Hence, there is a program $Q$ which, for every input $\vec{\alpha}$, a finite
sequence of ordinals, computes the characteristic function of a set $x$ of ordinals in a such a way that for every constructible $x\subseteq On$, there exists $\vec{\gamma}_{x}$ such that $Q$ computes the characteristic function
of $x$ on input $\vec{\gamma}_{x}$. We will use $Q$ to search through the constructible reals, looking for some $x\subseteq On$ such that $P^{x\cap\omega}\downarrow=1$. To do this, we use some natural enumeration $(\vec{\gamma}_{\iota}|\iota\in On)$
of finite sequences of ordinals and carry out the following procedure for each $\iota\in On$. First, find $\vec{\gamma}_{\iota}$, and let $x_{\iota}$ be the set of ordinals whose characteristic function is computed by $Q$ on input $\vec{\gamma}_{\iota}$. 
Then check, using $P$, whether $P^{x_{\iota}\cap\omega}\downarrow=1$. As $P^{y}\downarrow$ for all $y\subseteq\omega$, this will eventually be determined. If $P^{x_{\iota}\cap\omega}\downarrow=1$, then $x$ is found and we can write it on the tape.
Otherwise, continue with $\iota+1$. In this way, every constructible real will eventually be checked. By Shoenfield's absoluteness theorem, $x$ must be constructible, hence $x$ will at some point be considered, identified and written on the tape.
Thus $x$ is computable. 
\end{proof}

In fact, by almost the same reasoning, a much weaker assumption on $x$ is sufficient:

\begin{corollary}
 Let $x\subseteq\omega$ and $P$ be an $OTM$-program such that, for each $y\subseteq\omega$, we have $P^{y}\downarrow$ iff $y=x$ and $P^{y}\uparrow$, otherwise. Then $x$ is $OTM$-computable (without parameters).
\end{corollary}
\begin{proof}
First, observe that, by Shoenfield absoluteness, such an $x$ must be an element of $L$. Now, we use a slight modification of the proof of Theorem \ref{noOTMLoMe}: Again, we use a program $Q$ to successively write all constructible sets
of naturals to the tape. But now, we let $P$ run simultaneously on all the written reals. At some point, $x$ will be written to the tape and at some later point, $P$ will halt on it. When that happens, just copy the real on which $P$
halted to the beginning of the tape, thus writing $x$. This can then be used to decide every bit of $x$.
\end{proof}

\begin{rem} An easy reflection argument shows that a halting $OTM$- (and $ORM$-)computation with a real oracle always has a countable running time. Our results above hence in fact hold
for unresetting $\omega_1$-machines as well.\end{rem}

In the parameter-free case, this shows that, for extremely strong models of computation, the lost melody phenomenon is no longer present. This motivates a further inspection what exactly is necessary for the existence of lost melodies.

\section{$\alpha$-register machines}

Recall that, for $\alpha\in On$, let a resetting/unresetting $\alpha$-register machine works like an $ITRM$/$wITRM$ with the difference that a register may now contain an arbitrary ordinal $<\alpha$.
Hence, an $ITRM$ is a resetting $\omega$-register machine and a $wITRM$ is an unresetting $\omega$-register machine. This generalization was suggested at the end of \cite{wITRM}.

We denote by $wCOMP_{\alpha}$, $COMP_{\alpha}$, $wRECOG_{\alpha}$ and $RECOG_{\alpha}$ the set of reals computable by an unresetting $\alpha$-register machine, computable by a resetting $\alpha$-register machine,
recognizable by an unresetting $\alpha$-register machine and recognizable by a resetting $\alpha$-register machine, respectively.

We have seen that $wCOMP_{\omega}=wRECOG_{\omega}$, $COMP_{\omega}\subsetneq RECOG_{\omega}$, and that lost melodies for unresetting machines vanish when the register contents are unbounded. Hence, we ask:

\smallskip

\textbf{For which $\alpha$ are there lost melodies for resetting/unresetting $\alpha$-register machines?}

\smallskip

We start with the following easy observation:

\begin{lemma}
(1) Let $\alpha\geq\omega$. Then $wCOMP_{\alpha}\subseteq wRECOG_{\alpha}$ and $COMP_{\alpha}\subseteq RECOG_{\alpha}$.\\
(2) For all $\alpha$, we have $wCOMP_{\alpha}\subseteq COMP_{\alpha}$ and $wRECOG_{\alpha}\subseteq RECOG_{\alpha}$.
\end{lemma}
\begin{proof}
(1) As $\alpha\geq\omega$, we can again compute a real $x$ and compare it to the oracle bitwise.

(2) A terminating computation by an unresetting $\alpha$-machine will run exactly the same on a resetting $\alpha$-machine.
\end{proof}

\begin{lemma}{\label{bisimulation1}}
Let $\alpha>\beta$ be ordinals. Assume that there is an unresetting $\alpha$-program $P$ such that $P(b)\downarrow=1$ iff $b=\beta$ and $P(b)\downarrow=0$, otherwise.
Then $COMP_{\beta}\subseteq wCOMP_{\alpha}$.
\end{lemma}
\begin{proof}
Given $\alpha$, $\beta$ and $P$ as in the assumptions, let $y\in COMP_{\beta}$, and let $Q$ be a resetting $\beta$-program computing $y$.
To compute $y$ on an unresetting $\alpha$-machine, we describe how to simulate $Q$ on such a machine. Assume that $Q$ uses $k$ registers.
Reserve $k$ registers $R_{1}^{Q},...,R_{k}^{Q}$ of the unresetting $\alpha$-machine. Then, we proceed as follows: At successor steps,
simply carry out $Q$ on $R_{1}^{Q},...,R_{k}^{Q}$. At limit steps of the $Q$-computation, check, using $P$, whether any of these registers
contains $\beta$. If so, reset these register contents to $0$ and proceed, otherwise proceed without any modifications. This simulates $Q$
on an unresetting $\alpha$-machine.

To recognize limit steps in the computation of $Q$, reserve two extra registers, $R_1$ and $R_2$; initially, let $R_1$ contain $1$ and $R_2$
contain $0$. Whenever a step of $Q$ is carried out, swap their contents. Whenever their contents are equal, set $R_1$ again to $1$ and $R_2$ to $0$.
In this way, the contents of $R_1$ and $R_2$ will be equal iff the $Q$-computation has just reached a limit stage.
\end{proof}

\subsection{The unresetting case}

\begin{lemma}
 Let $\alpha<\beta$ be ordinals. Then $wCOMP_{\alpha}\subseteq wCOMP_{\beta}\subseteq wCOMP_{\omega_{1}}$ and $wRECOG_{\alpha}\subseteq wRECOG_{\beta}\subseteq wRECOG_{\omega_{1}}$.
\end{lemma}
\begin{proof}
 If $\alpha<\beta$, then terminating unresetting $\alpha$-computations work exactly the same on unresetting $\beta$-machines.
\end{proof}

We have seen above that $wCOMP_{\omega}=wRECOG_{\omega}$. We shall see now that that this happen again for $\omega_{1}$ and in fact for all but countably many countable ordinals $\alpha$.

\begin{lemma}{\label{unresettingomega1}}
 $wCOMP_{\omega_{1}}=wRECOG_{\omega_{1}}$.
\end{lemma}
\begin{proof}
This follows from Theorem \ref{noOTMLoMe}, as $wCOMP_{\omega_{1}}$ and $wRECOG_{\omega_{1}}$ are just the set of $ORM$-computable and $ORM$-recognizable reals (without ordinal parameters), respectively.
\end{proof}

\begin{thm}
 There is $\beta<\omega_{1}$ such that there are no lost melodies for unresetting $\gamma$-machines whenever $\gamma\geq\beta$.
\end{thm}
\begin{proof}
 Let $\beta$ be large enough such that $wCOMP_{\beta}=wCOMP_{\omega_{1}}$ and $wRECOG_{\beta}=wRECOG_{\omega_{1}}$. (This is possible by monotonicity and the fact that there are
only countably many programs.) Then, for all $\gamma\geq\beta$, we have $wCOMP_{\gamma}=wCOMP_{\omega_{1}}=wRECOG_{\omega_{1}}=wRECOG_{\gamma}$ by Lemma \ref{unresettingomega1}.
\end{proof}

Our next goal is to show that there are ordinals $\alpha$ for which $wCOMP_{\alpha}\subsetneq wRECOG_{\alpha}$, i.e. for which the lost melody phenomenon does occur:

\begin{lemma}{\label{recognizeomega}}
 There exists an unresetting $\omega+1$-program $P$ such that $P(x)\downarrow=1$ iff $x=\omega$ and $P(x)\downarrow=0$, otherwise.
\end{lemma}
\begin{proof}
Let $R_1$ be the register containing $x$. Use a register $R_2$ to successively count upwards from $0$. Use a flag to check whether the machine is in a limit state.
Eventually, the content of $R_1$ is reached. If this happens in a limit step, then $R_1$ contains $\omega$, otherwise, it does not.
\end{proof}

\begin{lemma}{\label{unresettingomegaplus1}}
 $wCOMP_{\omega+1}=COMP_{\omega}$ and $wRECOG_{\omega+1}=RECOG_{\omega}$, i.e. unresetting $\omega+1$-machines are equivalent in computational and recognizability strength to
$ITRM$s.
\end{lemma}
\begin{proof}(Sketch)
 One direction follows from Lemma \ref{recognizeomega} and Lemma \ref{bisimulation1}.

For the other direction, we show that a resetting $\omega$-machine (i.e. an $ITRM$) can simulate an unresetting $(\omega+1)$-machine. To see this, proceed as follows: Let $P$ be a program for an unresetting $\omega+1$-machine.
Assume that $P$ uses $k$ registers $R_{1}^{\prime},...,R_{k}^{\prime}$. We set up an $ITRM$-program in the following way: Reserve $R_{1},...,R_{k}$ for the simulation of $P$. In the simulation, let $0$ represent $\omega$
and let $i+1$ represent $i$ for all $i\in\omega\setminus\{0\}$. Whenever $P$ requires that the content of $R_{i}^{\prime}$ is set to the value $0$, set $R_{i}$ to $1$. When $P$ requires that the content of $R_{i}^{\prime}$ is
increased by $1$ and this content is currently $0$, stop. Otherwise, run $P$ on $R_{1},...,R_{k}$ in the usual way.
\end{proof}

\begin{thm}
 $wCOMP_{\omega+1}\neq wRECOG_{\omega+1}$, i.e. there are lost melodies for unresetting $\omega+1$-machines.
\end{thm}
\begin{proof}
This follows immediately from Lemma \ref{unresettingomegaplus1}, since, by the lost melody theorem for $ITRM$s, we have $COMP_{\omega}\neq RECOG_{\omega}$.
\end{proof}

\begin{rem} Arguments similar to the proof of Lemma \ref{unresettingomegaplus1} show that a resetting $\omega$-machine can in fact simulate an unresetting $(\omega+i)$-machine for all $i\in\omega$ (and more).
On the other hand, it can be shown that this is no longer the case for unresetting $\alpha$-machines when $\alpha>\omega_{\omega}^{CK}$ is exponentially closed: Coding $x\in L_{\alpha}$,
$x=\{y\in L_{\beta}|L_{\beta}\models\phi_{n}(y,\vec{\gamma})\}$ (where $\vec{\beta}$ is a finite sequence of ordinals and $\beta<\alpha$) by $(\alpha,n,\vec{\gamma})$ and using techniques similar to those developed in
\cite{KoSy}, we can evaluate arbitrary statements about the coded elements inside $L_{\alpha}$ with an unresetting $\alpha$-machine. This allows us to search through $\mathcal{P}(\omega)\cap L_{\alpha}$
with such a machine. As in the proof of Theorem \ref{noOTMLoMe}, one can now conclude that all reals in $L_{\alpha}$ recognizable by an unresetting $\alpha$-machine are already computable by such a machine.
We also saw that $RECOG_{\omega}\subseteq wRECOG_{\alpha}$ for $\alpha>\omega$. Now, the minimal real code $c:=cc(L_{\omega_{\omega}^{CK}})$ of $L_{\omega_{\omega}^{CK}}$ is an element of $L_{\omega_{\omega}^{CK}+2}$, and hence
of $L_{\alpha}$. $c$ is easily seen to be $ITRM$-recognizable, but as $c\notin L_{\omega_{\omega}^{CK}}$, it is not $ITRM$-computable. But $c\in RECOG_{\omega}\cap L_{\alpha}\subseteq wRECOG_{\alpha}\cap L_{\alpha}\subseteq wCOMP_{\alpha}$.
So $c\in wCOMP_{\alpha}-COMP_{\omega}$.\end{rem}

\textbf{Question}: We saw that $wCOMP_{\omega+1}\subsetneq wRECOG_{\omega+1}$ and there is a countable $\beta$ such that $wCOMP_{\gamma}=wRECOG_{\gamma}$ when $\gamma>\beta$. We do not know
if there are gaps in the ordinals for which lost melodies exist, i.e. if there are $\omega+1<\gamma<\delta$ such that $wCOMP_{\gamma}=wRECOG_{\gamma}$, but $wCOMP_{\delta}\subsetneq wRECOG_{\delta}$.

\subsection{The resetting case}

Note first that the computational strength for various values of $\alpha$ much higher than in the unresetting case:

\begin{thm}
Let $P_{i}$ be some natural enumeration of the $ORM$-programs. There is $\alpha<\omega_1$ such that some resetting $\alpha$-machine can solve the halting problem for parameter-free $ORM$s (i.e. unresetting $\omega_{1}$-machines), i.e. there is an
unresetting $\alpha$-program $Q$ such that $Q(i)\downarrow=1$ iff $P_{i}(0)$ stops and $Q(i)\downarrow=0$ iff $P_{i}(0)$ diverges.
\end{thm}
\begin{proof}
 Let $\alpha_{1}$ be the supremum of the register contents occuring in any halting parameter-free $ORM$-computation, let $\alpha_{2}$ be the supremum of the parameter-free
$ORM$-halting times and let $\alpha:=max\{\alpha_{1},\alpha_{2}\}$ (of course, as all registers are initially $0$ and a register content can be increased at most by $1$ in
one step, we will have $\alpha_{1}\leq\alpha_{2}$; it is not hard to see that in fact $\alpha_{1}=\alpha_{2}$).

Now consider $ORM$-programs with a fixed number $n$ of registers. Then a resetting $\alpha$-machine can solve the halting problem for such programs by simply simulating the given
program $P$ in the registers $R_{1},...,R_{n}$, while using a further register $R_{n+1}$ as a clock by increasing its content by $1$ whenever a step in the simulation is carried out.
If any of the registers $R_{1},...,R_{n},R_{n+1}$ overflows, then $P$ does not halt and we output $0$; otherwise, the simulation reaches the halting state and we output $1$.

A register overflow can be detected as follows: If a register $R$ has overflown, then the machine must be in a limit state (which can be detected by flags in the usual way) and
$R$ must contain $0$. In this situation, either there has been an overflow or the prior content of $R$ has been $0$ cofinally often in the current running time. This can be
distinguished by an extra register $R^{\prime}$ whose content is set to $0$ whenever $R$ contains $0$ and to $1$, otherwise. Hence, if $R^{\prime}$ contains $0$ in a limit state,
 then the content of $R$ must have been $0$ cofinally often.

Now, by \cite{KoSy}, there is a universal $ORM$, so we have an effective method how to find, for every $ORM$-program $P$, an $ORM$-program with the same halting behaviour, but using only $12$ registers. This, 
in combination with the halting problem solver for programs with any fixed number of registers, solves the halting problem for $ORM$s.
\end{proof}

The same holds when one considers the recognizability strength. To show this, we need some preliminaries.

\begin{lemma}{\label{longprograms}}
There is an $ITRM$-program $R$ such that, for each real $x$ coding an ordinal $\alpha<\omega_1$ according to $f:\alpha\rightarrow\omega$ injective, $R^{x}$ changes the content of the register $R_1$ exactly $\alpha+1$ many times.
\end{lemma}
\begin{proof}
By Lemma $2$ of \cite{KoMi}, the program $P$ defined there to test the oracle for well-foundedness of the coded relation runs for at least $\beta$ many steps when the oracle codes a well-ordering of length $\beta$.
Roughly, $P$ uses a stack to represent a finite descending sequence and attempts to continue it. We reserve a separate register $R_1$ and flip its content between $0$ and $1$ whenever a new element is put on the stack 
in the computation of $P^{x}$. The argument for Lemma $2$ of \cite{KoMi} shows that the content of $R_1$ will be changed at least $\alpha$ many times. If this happens exactly $\alpha$ many times, we simply set up
our program to flip the content of $R_1$ once more after $P$ has stopped. If it happens more than $\alpha$ many times, then some finite sequence $\vec{s}$ of natural numbers is the $\alpha+1$th sequence that is put
on the stack and we set up our program to stop once $\vec{s}$ has appeared on the stack.
\end{proof}

\begin{corollary}{\label{identifymachinetype}}
 Let $\alpha<\omega_1$. There is a resetting $\alpha$-program $I$ which, given a real $x$ coding an ordinal $\gamma$, halts with output $1$ iff $\gamma=\alpha$ and otherwise
halts with output $0$.
\end{corollary}
\begin{proof}
As $\alpha$-register machines can simulate $ITRM$s, we can use Lemma \ref{longprograms} to obtain a program $R$ that (run in the oracle $x$) changes the content of register $R_1$ exactly $\alpha+1$ many times. We use
a separate register $R_2$ that starts with content $0$ and whose content is incremented by $1$ whenever the content of $R_1$ is changed. Now, if $R_2$ overflows and the content of $R_1$ is changed
afterwards without $R$ halting, then $\alpha<\gamma$. If, on the other hand, $R$ stops without $R_2$ having overflown, we have $\gamma<\alpha$. If neither happens, i.e. if $R_2$ overflows and the
next change of the content of $R_1$ is followed by $R$ halting, then $\alpha=\gamma$. These scenarios are easy to detect.
\end{proof}

\begin{thm}{\label{resettingalpharecogstrength}}
There exists $\alpha<\omega_{1}$ and $x\subseteq\omega$ such that $x\in RECOG_{\alpha}$, but $x\notin wRECOG_{\omega_{1}}$.
\end{thm}
\begin{proof}
 Let $\tau$ be the supremum of stages containing new\\ $ORM$-recognizables. Let $\alpha+1>\tau$ be an index such that $L_{\alpha}\models ZF^{-}$ and let $r:=cc(L_{\alpha})$ be the $<_{L}$-minimal real coding $L_{\alpha}$. It is well known
that this implies $cc(L_{\alpha})\in L_{\alpha+2}$ (see e.g. \cite{BoPu}). 
Then $r$ is recognizable by a resetting $\alpha$-machine. To see this, first note that the property of being the minimal code of an index $L$-stage can be checked by an $ITRM$ using
the strategy described in the proof of the lost melody theorem for $ITRM$s in \cite{ITRM}. We saw above that resetting $\alpha$-machines can simulate $ITRM$s for all $\alpha\geq\omega$, hence
this can be carried out by a resetting $\alpha+1$-machine. It only remains to test whether the coded stage $L_{\zeta}$ is indeed $L_{\alpha}$. This can be done by using Corollary \ref{identifymachinetype} to
test whether the order type of $On\cap L_{\zeta}$ is $\alpha$.
\end{proof}

\begin{thm}{\label{alphahaltingtimes}}
 Let $\alpha<\omega_1$, and let $\delta>\alpha$ be such that $\delta$ is a limit of indices, but not itself an index. Then any $\alpha$-machine computation (with empty input and oracle) either halts in less than $\delta$ many steps
or does not halt at all.
\end{thm}
\begin{proof}
 This is an adaption of the argument given in \cite{KoMi} for $ITRM$s.
As $\delta$ is a limit of indices, but not an index, it follows (see e.g. \cite{Ch}, \cite{MaSr} or \cite{LePu}) that $L_{\delta}\models ZF^{-}+\forall{x}\exists{f}(f:\omega\rightarrow_{surj}x)$
 and hence that (see \cite{Je}) $\rho_{\omega}^{\delta}=\delta$, where $\rho_{\omega}^{\alpha}$ denotes the ultimate projectum of $L_{\alpha}$.
We claim that there is no $f:\xi\rightarrow\delta$ with unbounded range and $\xi<\delta$ definable over $L_{\delta}$. To see this, assume that there is such an $f$. By assumption,
there is, for every $\beta<\delta$ an index between $\beta$ and $\delta$ and hence $L_{\delta}$ contains a $<_{L}$-minimal bijection $g_{\beta}$ between $\omega$ and $\beta$.
Define a map $\bar{f}:\xi\times\omega\rightarrow_{surj}\delta$ via $\bar{f}(\iota,j)=g_{f(\iota)}(j)$. Let $h$ be a bijection between $\xi\times\omega$ and $\xi\omega$ and define 
$\tilde{f}:\xi\omega\rightarrow_{surj}\delta$ by $\tilde{f}:=\hat{f}\circ h^{-1}$. As $\xi<\delta$ and $L_{\delta}\models ZF^{-}$, we also have $\xi\omega<\delta$, and $\tilde{f}$ is
certainly definable over $L_{\delta}$. Hence a surjection from some $\zeta<\delta$ onto $\delta$ (and hence onto $L_{\delta}$ is definable over $L_{\delta}$, so that
$\rho_{\omega}^{\delta}<\delta=\rho_{\omega}^{\delta}$, a contradiction.

Now, there is a natural injection from the states of an $\alpha$-machine into $\alpha^{\omega}$, as the state can be given by a finite tuple of ordinals $<\alpha$ representing the register contents and
a single natural number representing the active program line. Such a map $j$ is definable over $L_{\alpha^{\omega}}$ and hence certainly an element of $L_{\delta}$.

Now let $P$ be an $\alpha$-program, and let $C$ be the computation of $P$, restricted to the first $\delta$ many steps. For a machine state $s$, let $\gamma_{s}$ denote $\text{sup}\{\beta<\delta|C(\beta)=s\}$.

Assume first that $\{\beta<\delta|\gamma_{C(\beta)}<\delta\}$ is cofinal in $\delta$, i.e. there are cofinally many states that appear only on boundedly many times. Then we can define, over $L_{\delta}$,
a partial map $f:\alpha\omega\rightarrow\delta$ by letting $f(\xi)=\gamma_{j^{-1}(\xi)}$ if $j^{-1}(\xi)$ is defined and $\gamma_{f^{-1}(\xi)}<\delta$ and otherwise $f(\xi)=0$. By assumption,
$f$ has unbounded range in $\delta$, which contradicts our observation above.

Hence, we may assume that there is some $\gamma<\delta$ such that every machine state assumed after time $\gamma$ appears at cofinally in $\delta$ many times. Suppose that $P$ uses
$n\in\omega$ many registers. The possible machine states are hence elements of $\times_{i=1}^{n}\alpha\times\omega$. Let us partially order the set $S$ of machine states occuring
in the computation after time $\gamma$ by letting
$(\beta_{1},...,\beta_{n},k)\leq_{s}(\gamma_{1},...,\gamma_{n},l)$ iff $k\leq l$ and $\beta_{i}\leq\gamma_{i}$ for all $i\in\{1,...,n\}$. It is easy to see that $\leq_{s}$ is well-founded.

For each two states $Z_{1},Z_{2}\in S$, there is $Z_{3}\in S$ such that $Z_{1}\leq_{s}Z_{3}$ and $Z_{1}\leq_{s}Z_{3}$: To see this, observe that we can define over $L_{\delta}$ a strictly increasing
map $\sigma:\omega\rightarrow\delta$ such that $C(\sigma(2i))=Z_{1}$ and $C(\sigma(2i+1))=Z_{2}$ for all $i\in\omega$. By our observation above, $\text{rng}(\sigma)$ must be bounded in $\delta$, so let
$\bar{\delta}:=\text{sup rng}(\sigma)$. Then $C(\bar{\delta})$ is as desired.

Now, by well-foundedness of $\leq_{s}$, $S$ must contain a minimal element $Z$. It is easy to see that $Z$ is in fact unique: For if $Z_{1}$ and $Z_{2}$ were two distinct minimal elements of $S$,
then by our last observation, we would have $Z_{3}\in S$ with $Z_{3}\leq Z_{1}$ and $Z_{3}\leq Z_{2}$. As $Z_{1}\neq Z_{2}$, one of the inequalities would have to be strict, contradicting the minimality
of $Z_{1}$ and $Z_{2}$.

Hence $Z$ is assume cofinally in $\delta$ many times, while all other states occuring after time $\gamma$ are $\geq_{s}Z$. Consequently, the machine state at time $\delta$ is again $Z$ and it is easy
to see that the computation cycles. Hence, a resetting $\alpha$-machine computation either halts before time $\delta$ or does not halt at all.

\end{proof}

\begin{corollary}{\label{alphacomputability}}
$COMP_{\alpha}\subseteq L_{\delta}$, where $\delta$ is the minimal limit of indices above $\alpha$ which is not itself an index.
\end{corollary}
\begin{proof}
Since $\delta$ is not an index, every subset of $\omega$ definable over $L_{\delta}$ is an element of $L_{\delta}$. Now let $x\in COMP_{\alpha}$, and let $P$ be a resetting $\alpha$-program that computes $x$,
i.e. $P(i)\downarrow=1$ if $i\in x$ and $P(i)\downarrow=0$ if $i\notin x$ for all $i\in\omega$. By Theorem \ref{alphahaltingtimes} and as $P(i)\downarrow$ for all $i\in\omega$, the halting time of $P(i)$ must
be smaller than $\delta$ for all $i\in\omega$. Hence $i\in x$ is expressed over $L_{\delta}$ by an $\in$-formula stating the existence of a halting $P$-computation with input $i$ and output $1$. Consequently,
we must have $x\in L_{\delta}$.
\end{proof}

This allows us to show that there are lost melodies for resetting $\alpha$-machines for all infinite $\alpha<\omega_1$:

\begin{thm}{\label{alpharesettinglomes}}
 Let $\alpha<\omega_1$ be infinite. Then there $COMP_{\alpha}\neq RECOG_{\alpha}$, i.e. there is a lost melody for resetting $\alpha$-machines.
\end{thm}
\begin{proof}
Given $\alpha<\omega_1$, let $r_{\alpha}$ be the $<_{L}$-minimal real coding an $L$-level $L_{\gamma}$ such that $\gamma$ is a limit of indices but not itself an index, $\gamma+1$ is an index and $L_{\gamma}$ contains
a real coding $\alpha$. Then we must also have $r_{\alpha}\notin COMP_{\alpha}$
by Corollary \ref{alphacomputability}. We show that $r_{\alpha}\in RECOG_{\alpha}$ by an argument similar to the proof of the lost melody theorem for $ITRM$s. Let $x$ be given in the oracle. 
First, we can - even with an $ITRM$ - check whether $x$ codes an $L$-level $L_{\zeta}$ with cofinally many indices. If not, $x\neq r_{\alpha}$. If so, the methods developed in the proof of the lost melody theorem for $ITRM$s
allow us to compute from $x$ the truth predicate for $L_{\zeta+2}$, which allows us to check whether $\zeta$ and $\zeta+1$ are indices. If $\zeta$ is an index or $\zeta+1$ is not, then $x\neq r_{\alpha}$. Otherwise,
we need to check whether $\zeta>\alpha$ (this suffices to guarantee the existence of a real coding $\alpha$, since at this point we already know that $\zeta$ is a limit of indices). This can be done as follows: Inside $r_{\alpha}$,
$\alpha$ must be coded by some natural number $i$ that can be given to our program in advance. So we test whether $i$ codes an ordinal $\theta$ in $x$. If not, then $x\neq r_{\alpha}$. 
Now, we can easily compute from $i$ and $x$ a real $y$ coding the order type $\theta$ (just delete every $p(k,j)\in x$ with $\{p(k,i),p(j,i)\}\not\subseteq x$) and then use Corollary \ref{identifymachinetype} to check whether $y$ codes $\alpha$. If not, then $x\neq r_{\alpha}$.
Otherwise, we know that $i$ codes $\alpha$ inside $r_{\alpha}$.

Next, we check whether there is any $\alpha<\zeta^{\prime}<\zeta$ with the same properties. If yes, then $x\neq r_{\alpha}$. Otherwise, we know that $x$ codes $L_{\gamma}$ and it remains to check the $<_{L}$-minimality of $x$.
 As $L_{\zeta+1}$ is an index, we know that the minimal real coding $L_{\zeta}$ must be an element of $L_{\zeta+2}$. As we just mentioned, we can, given $x$, evaluate the truth predicate for $L_{\zeta+2}$. Hence, we can 
search through (the code of) $L_{\zeta+2}$ until we find the $<_{L}$-minimal real coding $L_{\zeta}$ and compare it with $x$. If these reals disagree, then $x\neq r_{\alpha}$, otherwise $x=r_{\alpha}$.
So $r_{\alpha}$ is recognizable.

This proves that $r_{\alpha}$ is a lost melody for resetting $\alpha$-machines.

It remains to see that such an $L$-level $L_{\gamma}$ exists. To see this, let $\gamma>\alpha$ be a a minimal limit of indices, and let $\alpha<\delta<\gamma$ be an index. Let $x$ be a real
such that $x\in L_{\delta+1}-L_{\delta}$. Then the elementary hull $H$ of $\{x\}$ in $L_{\gamma}$ is (isomorphic to) an $L$-level $L_{\beta}$ with cofinally many indices which contains $x$, where $\beta\leq\gamma$.
It follows that $\beta=\gamma$ and that in fact $H=L_{\gamma}$. This hull is definable over $L_{\gamma+1}$, so that we get a bijection between $\omega$ and $L_{\gamma}$ in $L_{\gamma+2}$ by the standard finestructural arguments.
Hence $\gamma+1$ is indeed an index, so $\gamma$ is as desired.
\end{proof}

\begin{rem}
 Note that, as parameter-free computations have countable length, $wCOMP_{\omega_{1}}$ corresponds to parameter-free $ORM$-computability. Moreover, by Shoenfield absoluteness,
we have $wRECOG_{\omega_{1}}\subseteq wRECOG_{\omega_{1}}\subseteq\mathcal{P}^{L}(\omega)$. Consequently, if $P$ is an $ORM$-program recognizing $x\subseteq\omega$, then $P$, now interpreted as a program
for an unresetting $\omega_{1}$-ITRM, recognizes $x$ as well: The computations will only take countable many steps and hence no limit of register contents can exceed $\omega_{1}$, prompting an overflow.
Hence $wRECOG_{\omega_{1}}$ coincides with the set of $ORM$-recognizable reals.
The same holds for every $\alpha\geq\omega_{1}$.
As $ORM$-computability and $ORM$-computability coincide, there are no lost melodies for unresetting $\alpha$-ITRMs with $\alpha\geq\omega_{1}$.
\end{rem}



\section{Conclusion and further work}

We have seen that lost melodies exist for a resetting $\alpha$-machines iff $\alpha<\omega_1$ is infinite and that for unresetting $\alpha$-machines, lost melodies do not exist for $\alpha=\omega$,
do exist for $\alpha=\omega+1$ and cease to exist from some countable ordinal on. In the special case of resetting $\omega$-machines or $ITRM$s, the recognizable allow for a detailed analysis among the
constructible reals and show several surprising regularities. In the parameter-$OTM$-case, we reach the limits of $ZFC$.
In general, the relation between the computability and recognizability strength of transfinite models of computation seems to be far from trivial.

In this paper, we have restricted our attention to reals, as these can be dealt with by all models in question and can hence be used as a basis
for comparison. One could equally well consider subsets of other ordinals, which might be more appropriate for some models.

Once we do this, interesting questions arise, even for classical Turing machines: Consider, for example, Turing programs using at most $n$ states and
symbols for some $n\in\mathbb{N}$. Let us say that a natural number $k$ is $n$-computable iff there is such a Turing program that outputs $k$ when run 
on the empty input, and let us say that $k$ is $n$-recognizable iff there is such a Turing program that stops with output $1$ on the input $k$
and with output $0$ on all other integers. Are there infinitely many $n\in\mathbb{N}$ for which there exists $l\in\mathbb{N}$ which is $n$-recognizable, but
not $n$-computable? This provides a kind of a miniaturization of the question for the existence of lost melodies.

Another topic one might pursue is to consider the various generalizations of Turing machines ($ITTM$s, $\alpha$-Turing machines, $\alpha$-$\beta$-Turing machines).

\section{Acknowledgments}

We are indebted to Philipp Schlicht for sketching a proof of Lemma \ref{relbar}, a crucial hint for the proof of Theorem \ref{wsigma1}
and suggesting several very helpful references, in particular concerning Theorem \ref{0sharprecog}. We also thank the anonymous referee for several corrections and suggestions that helped
to considerably improve the paper.


\begin{thebibliography}{}
\bibitem[Ba]{Ba} J. Barwise. Admissible Sets and Structures. Springer (1975)
\bibitem[BoPu]{BoPu} G. Boolos, H. Putnam. Degrees of unsolvability of constructible sets of integers. Journal of Symbolic Logic, 37 (1972), 81-89
\bibitem[Ca]{Ca} M. Carl. The distribution of $ITRM$-recognizable reals. To appear in: Annals of Pure and Applied Logic, special issue from the Turing Centenary Conference CiE 2012: How the World Computes 
\bibitem[Ca2]{Ca2} M. Carl. Optimal results on $ITRM$-recognizability. Preprint. arXiv:1306.5128v1 [math.LO] 
\bibitem[Ch]{Ch} C.T. Chong. A recursion-theoretic characterization of constructible reals. Bulletin of the London Mathematical Society 9,  241-244 (1977)
\bibitem[Ko]{Ko} P. Koepke. Computing a model of set theory. In New Computational Paradigms. S. Barry Cooper et al, eds., Lecture Notes in Computer Science 3988 (2006), 223-232
\bibitem[Co]{Co} S. Coskey. Infinite-Time Turing Machines and Borel Reducibility. Mathematical Theory and Computational Practice (2009) 
\bibitem[Cu]{Cu} N. Cutland. Computability. An introduction to recursive function theory. Cambridge University Press (1980)
\bibitem[FrWe]{FrWe} S.D. Friedman, P. D. Welch. J. Symbolic Logic Volume 76, Issue 2 (2011), 620-636. 
\bibitem[GHJ]{GHJ} V. Gitman, J.D. Hamkins, Th. A. Johnstone. What is the theory $ZFC$ without power set? arXiv:1110.2430 [math.LO]
\bibitem[HaLe]{HaLe} J. D. Hamkins, A. Lewis. Infinite Time Turing Machines. Journal of Symbolic Logic 65(2), 567-604 (2000)
\bibitem[HMSW]{HMSW} J. D. Hamkins, R. Miller, D. Seabold, and S. Warner. Infinite time computable model theory. , S. B. \. Cooper, B. L\"owe, and A. Sorbi, Ed., New York: Springer, 2008, pp. 521-557.
\bibitem[ICTT]{ICTT} M. Carl. Towards a Church-Turing-Thesis for Infinitary Computations. Electronic Proceedings of CiE 2013.
\bibitem[ITRM]{ITRM} M. Carl, T. Fischbach, P. Koepke, R. Miller, M. Nasfi, G. Weckbecker. The basic theory of infinite time register machines. Archive for Mathematical Logic 49 (2010) 2, 249-273
\bibitem[Je]{Je} R.B. Jensen. The fine structure of the constructible hierarchy. Annals of Mathematical Logic 4 (1972) 
\bibitem[JeKa]{JeKa} R. Jensen, C. Karp. Primitive Recursive Set Functions. In: Proceedings of Symposia in Pure Mathematics, Volume XIII, Part 1 (1971)
\bibitem[Ka]{Ka} A. Kanamori. The higher infinite. Springer $2005$.
\bibitem[Ko]{Ko} P. Koepke. Computing a model of set theory. In New Computational Paradigms. S. Barry Cooper et al, eds., Lecture Notes in Computer Science 3988 (2006), 223-232
\bibitem[KoSe1]{KoSe1} P. Koepke, B. Seyfferth. Ordinal machines and admissible recursion theory. Annals of Pure and Applied Logic, 160 (2009), 310-318.
\bibitem[KoSe2]{KoSe2} P. Koepke, Benjamin Seyfferth. Towards a theory of infinite time Blum-Shub-Smale machines. 7 pages. Proceedings of CiE 2012
\bibitem[KoMi]{KoMi} P. Koepke, Russell Miller.  An enhanced theory of infinite time register machines. In Logic and Theory of Algorithms. A. Beckmann et al, eds., Lecture Notes in Computer Science 5028 (2008), 306-315
\bibitem[KoSy]{KoSy} P. Koepke, R. Syders. Register Computations on Ordinals. Archive for Mathematical Logic 47 (2008), 529-548
\bibitem[KoWe]{KoWe} P. Koepke, P. Welch. A generalised dynamical system, infinite time register machines, and $Pi^{1}_{1}-CA_0$. In CiE 2011. B. L\"owe et al, eds., Lecture Notes in Computer Science 6735 (2011), 152-159.
\bibitem[LePu]{LePu} S. Leeds, H. Putnam. An intrinsic characterization of the hierarchy of constructible sets of integers. Logic Colloquium '69 (North-Holland, Amsterdam, 1971)
\bibitem[MaSr]{MaSr} W. Marek, M. Srebrny. Gaps in the constructible universe. Ann. Math. Logic, 6 (1974), 359-394
\bibitem[Ma]{Ma} A.R.D. Mathias. Provident sets and rudimentary set forcing. Preprint, available at https://www.dpmms.cam.ac.uk/~ardm/
\bibitem[ORM]{ORM} P. Koepke. R. Syders. Computing the recursive truth predicate on ordinal register machines. In Logical Approaches to Computational Barriers, Arnold Beckmann et al., eds., Computer Science Report Series 7 (2006), Swansea, 160-169
\bibitem[OTM]{OTM} P. Koepke. Turing computations on ordinals. Bulletin of Symbolic Logic 11 (2005), 377-397
\bibitem[Sa]{Sa} G. Sacks. Higher recursion theory. Springer 1990.
\bibitem[Sa2]{Sa2} G. Sacks. Countable admissible ordinals and hyperdegrees. Advances in Mathematics 19, 213-262 (1976)
\bibitem[Se]{Se} B. Seyfferth. Three models of ordinal computability, PhD thesis, University of Bonn, 2012
\bibitem[SeSc]{SeSc} P. Schlicht, B. Seyfferth. Tree representations via ordinal machines. Computablility 1, 1 (2012), 45-57
\bibitem[wITRM]{wITRM} P. Koepke. Infinite Time Register Machines. Logical Approaches to Computational Barriers, Arnold Beckmann et al., eds., Lecture Notes in Computer Science 3988 (2006), 257-266

\end{thebibliography}
\end{document}